%% file: RamseyDegrees2.tex
\documentclass{amsart}
\usepackage{amsmath,amsfonts,amssymb,amsthm, enumerate,latexsym,mathrsfs, color}
\usepackage[alphabetic]{amsrefs}
\usepackage{graphicx}
\usepackage[all]{xy}

\author{Lionel Nguyen Van Th\'e}

\address{Aix Marseille Univ, CNRS, Centrale Marseille, I2M UMR 7373, 13453 Marseille, France}

\email{lionel.nguyen-van-the@univ-amu.fr}

\thanks{This work has been partially supported by the GrupoLoco project (ANR-11-JS01-0008) funded by the French Government, managed by the French National Research Agency (ANR)}



\date{May 2017}

\title[Ramsey degrees]{Finite Ramsey degrees and Fra\"iss\'e expansions with the Ramsey property}

\input{Macros15Nov}

\begin{document}
\maketitle

\begin{abstract}
By a result of Zucker, every Fra\"iss\'e structure $\m F$ for which the elements of $\Age(\m F)$ have finite Ramsey degrees admits a Fra\"iss\'e precompact expansion $\m F^{*}$ whose age $\Age(\m F^{*})$ has the Ramsey property. While the original method uses dynamics in spaces of ultrafilters, the purpose of the present short note is to provide a different proof, based on classical tools from Fra\"iss\'e theory.  

\end{abstract}

\section{Introduction}

Recent works in structural Ramsey theory have largely been influenced by the paper \cite{Kechris05} of Kechris-Pestov-Todorcevic which exhibited a strong connection with topological dynamics. When studying it, it became apparent that of particular interest are those Fra\"iss\'e structures $\m F$ whose elements of $\Age(\m F)$ have a finite Ramsey degree. (For background and notations, see Section \ref{section:background}.) On the surface, this property is only a slight weakening of the now classical Ramsey property introduced by Ne\v set\v ril and R\"odl. However, practice suggests that it could be substantially less restrictive, and knowing exactly to which extent this is so remains open, despite the recent progress made on the model theoretic side (for example, Evans proved that $\omega$-categoricity of $\m F$ is not sufficient \cite{Evans15}), on the combinatorial side (see \cite{Hubicka16} by Hubi\v cka and Ne\v set\v ril, where the scope of the main technique to prove the Ramsey property is made broader still) and on the dynamical side (see \cite{Zucker16} by Zucker, \cite{Melleray16} by Melleray-Tsankov-Nguyen Van Th\'e and \cite{BenYaacov16c} by Ben Yaacov-Melleray-Tsankov). 

One method to make sure that a Fra\"iss\'e structure $\m F$ has finite Ramsey degrees is to construct a precompact expansion $\m F^{*}$ so that $\Age(\m F^{*})$ has the Ramsey property. It turns out that \emph{every} Fra\"iss\'e structure with finite Ramsey degrees arises that way: 

\begin{thm***}[Zucker \cite{Zucker16}]
\label{thm:main}
Let $\m F$ be a Fra\"iss\'e structure. TFAE: 
\begin{enumerate}
\item[i)] Every element of $\Age(\m F)$ has a finite Ramsey degree. 
\item[ii)] The structure $\m F$ admits a Fra\"iss\'e precompact expansion $\m F ^{*}$ whose age has the Ramsey property, the expansion property relative to $\Age(\m F)$, and is made of rigid elements. 
\end{enumerate} 
\end{thm***}

The original proof of this result uses topological dynamics in spaces of ultrafilters. The purpose of the present short note is to provide a different proof, based on classical tools from Fra\"iss\'e theory as well as some additional results from \cite{NVT13a}.


\section{Background and notation}

\label{section:background}

The purpose of this section is to introduce the terminology that will be used in order to show the main result of the paper. It is based on \cite{NVT13a}, which itself rests on \cite{Kechris05}. The reader will be assumed to have some familiarity with Fra\"iss\'e theory. Given a first order language $L$, the \emph{age} of a countable $L$-structure $\m F$ is the class $\Age(\m F)$ of all of its finite substructures (up to isomorphism). The structure $\m F$ itself is a \emph{Fra\"iss\'e structure} when it is countable, every finite subset of $\m F$ generates a finite substructure of $\m F$ ($\m F$ is \emph{locally finite}), and every isomorphism between finite substructures of $\m F$ extends to an automorphism of $\m F$ ($\m F$ is  \emph{ultrahomogeneous}). Given structures $\m A$, $\m B$, $\m C$ in that language, write $\m A \cong \m B$ when $\m A$ and $\m B$ are isomorphic, and define the set of all copies of $\m A$ in $\m C$ as 

$$ \binom{\m C}{\m A} := \{ \mc A \subset \m C : \mc A \cong \m A\}.$$  

Given a class $\mathcal K$ of $L$-structures, $\m A \in \mathcal K$ has a \emph{finite Ramsey degree} (in $\mathcal K$) when there exists an integer $t(\m A)$ such that for every $k\in \N$, every $\m B\in \mathcal K$, there exists $\m C\in \mathcal K$ such that for every map $\chi: \binom{\m C}{\m A}\rightarrow [k]:=\{0, 1,... , k-1\}$ (usually referred to as a $k$-\emph{coloring of} $\binom{\m C}{\m A}$) there is $\mc B\in \binom{\m C}{\m B}$ such that $\chi$ takes at most $t(\m A)$-many colors on $\binom{\mc B}{\m A}$. When $\mathcal K =\Age(\m F)$ for some Fra\"iss\'e structure $\m F$, this is equivalent to the fact that for every $k$-coloring $\chi$ of $\binom{\m F}{\m A}$, there is $\mc B\in \binom{\m F}{\m B}$ such that $\chi$ takes at most $t(\m A)$-many colors on $\binom{\mc B}{\m A}$. When in addition $t(\m A)=1$ for every $\m A\in \mathcal K$, $\mathcal K$ has the \emph{Ramsey property}.

Consider now an expansion $L^{*}$ of $L$ with at most countably many relational symbols. Say that an expansion $\mathcal K^{*}$ of $\mathcal K$ in $L^{*}$ is \emph{precompact} when any $\m A \in \mathcal K$ only has finitely many expansions in $\mathcal K^{*}$. Similarly, say that a Fra\"iss\'e structure $\m F^{*}$ is a precompact expansion of $\m F$ when $\Age(\m F^{*})$ is a precompact expansion of $\Age(\m F)$. Finally, say that $\mathcal K ^{*}$ satisfies the \emph{expansion property} relative to $\mathcal K$ if, for every $ \m A \in \mathcal{K}$, there exists $ \m B \in \mathcal{K}$ such that every expansion of $\m A$ in $\mathcal K^{*}$ embeds in every expansion of $\m B$ in $\mathcal K^{*}$.

\section{Proof of the main result}

\label{section:proof}

With the terminology of the previous section in mind, we are now ready to provide a proof of Zucker's theorem using classical techniques. The implication $ii)\Rightarrow i)$ is obvious so we concentrate on $i)\Rightarrow ii)$. The first step makes use of an ingenious technique due to K\v{r}\'i\v{z} in \cite{Kriz91}: 

\begin{defn}
Let $\m A, \m B \in \Age(\m F)$. 

For $t\in \N$, say that $\m B$ is \emph{$t$-$\m A$-Ramsey} when for every $k\in \N$, and every $k$-coloring $\chi$ of $\binom{\m F}{\m A}$, there exists $\mc B \in \binom{\m F}{\m B}$ such that $\chi$ takes at most $t$-many values on $\binom{\mc B}{\m A}$.   

For $E$ an equivalence relation on $\binom{\m B}{\m A}$, say that $\m B$ is \emph{$E$-$\m A$-Ramsey} when for every $k\in \N$, $k$-coloring $\chi$ of $\binom{\m F}{\m A}$, there exists an embedding $b : \m B\rightarrow \m F$ such that: $$\forall \mc A_{0}, \mc A_{1}\in \binom{\m B}{\m A} \quad \mc A_{0}E\mc A_{1} \Rightarrow \chi(b(\mc A_{0}))=\chi(b(\mc A_{1}))$$ 
\end{defn}

\begin{lemma}
\label{lemma:1}
Let $\m A, \m B \in \Age(\m F)$, $t\in \N$. Assume that $\m B$ is $t$-$\m A$-Ramsey. Then $\m B$ is $E$-$\m A$-Ramsey for some equivalence relation $E$ on $\binom{\m B}{\m A}$ with at most $t$-many classes. 
\end{lemma}

\begin{proof}

Let $\mathcal E$ denote the (finite) set of all equivalence relations on $\binom{\m B}{\m A}$ with at most $t$-many classes. We assume that $\m B$ is not $E$-$\m A$-Ramsey for any $E\in \mathcal E$, and show that $\m B$ is not $t$-$\m A$-Ramsey. 

By assumption, for every $E\in \mathcal E$, there exists $k_{E}\in \N$ and $\chi_{E}:\binom{\m F}{\m A}\rightarrow [k_{E}]$ witnessing that  $\m B$ is not $E$-$\m A$-Ramsey, i.e. for every embedding $b : \m B\rightarrow \m F$, there exist $\mc A_{0}E\mc A_{1}$ in $\binom{\m B}{\m A}$ such that $\chi_{E}(b(\mc A_{0}))\neq \chi_{E}(b(\mc A_{1}))$. 

Let $\displaystyle \chi : \binom{\m F}{\m A}\rightarrow \prod_{E\in \mathcal E}[k_{E}]$ be the finite coloring defined by $\chi(\mc A)=(\chi_{E}(\mc A))_{E\in \mathcal E}$. 

We claim that for every $\mc B \in \binom{\m F}{\m B}$, $\chi$ takes at least $(t+1)$-many values on $\binom{\mc B}{\m A}$. Let $b : \m B \rightarrow \m F$ an embedding such that $\mc B=b(\m B)$. Define $E_{b}$ on $\binom{\m B}{\m A}$ by: $$ \mc A_{0}E_{b}\mc A_{1} \Leftrightarrow \chi(b(\mc A_{0}))=\chi(b(\mc A_{1}))$$ 

By construction of $\chi$, $E_{b}$ disagrees with each $E\in \mathcal E$ on at least one pair of elements of $\binom{\m B}{\m A}$. It follows that $E_{b} \notin \mathcal E$, as required. \end{proof}

Writing $t(\m A)$ the Ramsey degree in $\Age(\m F)$, it follows that every $\m B \in \Age(\m F)$ is $E_{\m A} ^{\m B}$-$\m A$-Ramsey for some equivalence relation $E_{\m A} ^{\m B}$ on $\binom{\m B}{\m A}$ with at most $t(\m A)$-many classes.   

\begin{lemma}
For every $\m A \in \Age(\m F)$, there is an equivalence relation $E_{\m A}$ on $\binom{\m F}{\m A}$ with at most $t(\m A)$-many classes such that every finite substructure $\m B$ of $\m F$ is $(\restrict{E_{\m A}}{B})$-$\m A$- Ramsey. (where $(\restrict{E_{\m A}}{B})$ denotes the restriction of $E_{\m A}$ to $B$.)
\end{lemma}

\begin{proof}
Write $\m F =\{ x_{k} : k\in \N\}$ and let $\m B_{n}$ denote the substructure of $\m F$ supported by $\{ x_{k} : k\leq n\}$. By Lemma \ref{lemma:1}, each $\m B_{n}$ is $E_{\m A} ^{\m B_{n}}$-$\m A$-Ramsey for some equivalence relation $E_{\m A} ^{\m B_{n}}$ on $\binom{\m B_{n}}{\m A}$ with at most $t(\m A)$-many classes. Observe that if $m\leq n$, then $\m B_{m}$ is also $(\restrict{E_{\m A} ^{\m B_{n}}}{B_{m}})$-$\m A$- Ramsey. Thus, the set $\{ (\m B_{m}, \restrict{E_{\m A} ^{\m B_{n}}}{B_{m}})) : m\leq n\}$ becomes a finitely branching tree when equipped with the relation $$ (\m B, E)\leq (\m C, F)\Leftrightarrow \restrict{F}{\binom{\m B}{\m A}}=E$$ 

By K\"onig's lemma, this tree admits an infinite branch, whose union is of the form $(\m F, E_{\m A})$. The relation $E_{\m A}$ as required. \end{proof}

From that point on, the proof makes use of several results from \cite{NVT13a}. For each $\m A\in \Age(\m F)$, we now add predicates, $P_{\m A, 0},..., P_{\m A, t(\m A)-1}$, one for each equivalence class of $E_{\m A}$. Write $\vec P$ for the family $(P_{\m A, i})_{\m A, i}$ where $\m A$ ranges over $\Age(\m F)$ and $i<t(\m A)$. Note that every element of $\Age(\m F)$ only has finitely many expansions in $\Age(\m F, \vec P)$. Next, let $<$ be a linear ordering on $\m F$, and consider the logic action of $\Aut(\m F)$ on the compact space $$\prod_{\stackrel{\m A\in \Age(\m F)}{i<t(\m A)}}[2]^{\m F ^{|\m A|}}\times \mathrm{LO}(\m F)$$

Pick $(\vec P ^{*}, <^{*})\in \overline{\Aut(\m F)\cdot (\vec P, <)}$ with minimal orbit closure in that space. Then, $<^{*}$ is a linear ordering on $\m F$, and writing $\vec P^{*}=(P^{*} _{\m A, i})_{\m A, i}$, the family $(P^{*}_{\m A, i})_{i<t(\m A)}$ is a partition of $\binom{\m F}{\m A}$ for every $\m A\in \Age(\m F)$. Let $\mathcal K^{*}:=\Age(\m F, \vec P^{*}, <^{*})$. 

\begin{lemma}
Let $\m A\in \Age(\m F)$. Then $\m A$ has at least $t(\m A)$-many expansions in $\mathcal K^{*}$. 
\end{lemma}

\begin{proof}
Let $\m B \in \Age(\m F)$ witness the fact that $t(\m A)$ is the Ramsey degree of $\m A$ in $\Age(\m F)$. This means that there is $\chi : \binom{\m F}{\m A}\rightarrow [k]$ taking at least $t(\m A)$-many values on $\binom{\mc B}{\m A}$ whenever $\mc B\in \binom{\m F}{\m B}$. Let $\mc B \in \binom{\m F}{\m B}$. Because $\vec P^{*}\in \overline{\Aut(\m F)\cdot \vec P}$, we have $\Age(\m F, \vec P^{*})\subset \Age(\m F, \vec P)$ (see \cite{NVT13a}*{Proposition 6}). It follows that if $E$ denotes the equivalence relation induced on $\binom{\mc B}{\m A}$ by $(P^{*}_{\m A, i})_{i<t(\m A)}$, then $\mc B$ is $E$-$\m A$-Ramsey. Hence, there is an embedding $b:\mc B\rightarrow \m F$ such that the value of $\chi(b(\mc A))$ depends only on the $E$-class of $\mc A$ in $\binom{\mc B}{\m A}$. By choice of $\chi$, it follows that $\binom{\mc B}{\m A}$ is made up of at least $t(\m A)$ many $E$-classes. In other words, at least $t(\m A)$-many expansions of $\m A$ appear in the substructure of $(\m F, \vec P)$ supported by $\mc B$. This is valid for every copy $\mc B$ of $\m B$ in $\m F$, so every expansion of $\m B$ in $\Age(\m F, \vec P)$ contains at least $t(\m A)$-many expansions of $\m A$. As a result, this is also the case for every expansion of $\m B$ in $\mathcal K^{*}$, and $\m A$ has at least $t(\m A)$-many expansions in $\mathcal K^{*}$. \end{proof}

Clearly, $\mathcal K^{*}$ has the hereditary property and the joint embedding property. It has the expansion property relative to $\Age(\m F)$ because $(\vec P ^{*}, <^{*})$ has minimal orbit closure (see \cite{NVT13a}*{Theorem 4}). By the previous lemma, every $\m A \in \Age(\m F)$ has a finite Ramsey degree in $\Age(\m F)$ whose value is at most the number of non-isomorphic expansions of $\m A$ in $\mathcal K^{*}$. According to \cite{NVT13a}*{Proposition 8}, these conditions guarantee that $\mathcal K^{*}$ has the Ramsey property, and is therefore a Fra\"iss\'e class. As such, it has a Fra\"iss\'e limit, which we write $\m F^{*}$. To finish the proof, it suffices to show that $\m F^{*}$ is an expansion of $\m F$.

\begin{lemma}
The structure $\m F^{*}$ is an expansion of $\m F$. 
\end{lemma}

\begin{proof}
It suffices to show that the reduct $\restrict{\m F^{*}}{L}$ of $\m F^{*}$ to the language $L$ of $\m F$ is Fra\"iss\'e. Indeed, $\Age(\restrict{\m F^{*}}{L})=\Age(\m F)$, so $\restrict{\m F^{*}}{L}=\m F$ will follow. 

To show that $\restrict{\m F^{*}}{L}$ is Fra\"iss\'e, following  \cite{Kechris05}*{5.2}, it suffices to show that $\mathcal K^{*}$ has the following \emph{reasonability} property: For every embedding $\pi : \m A \rightarrow \m B$ between elements of $\Age(\m F)$, and every expansion $\m A^{*}$ of $\m A$ in $\mathcal K^{*}$, there exists an expansion $\m B^{*}$ of $\m B$ in $\mathcal K^{*}$ so that $\pi$ induces an embedding from $\m A^{*}$ to $\m B^{*}$. 

Let us first check that this condition is sufficient. Let $\pi : \m A \rightarrow \m B$ be an inclusion embedding between finite substructures of $\restrict{\m F^{*}}{L}$, and $\varphi : \m A\rightarrow \restrict{\m F^{*}}{L}$ be an embedding. Then $\varphi(\m A)$ supports a substructure of $\m F^{*}$, and this can be pulled back to $\m A$ to define an expansion $\m A^{*}$ of $\m A$ in $\mathcal K^{*}$. By reasonability, expand $\m B$ to $\m B^{*}\in \mathcal K^{*}$ so that $\pi$ induces an embedding from $\m A^{*}$ to $\m B^{*}$, as required. Because $\m F^{*}$ is Fra\"iss\'e, $\pi$ extends to $\m B^{*}$, which induces $\phi:\m B \rightarrow \m F$ extending $\varphi$. 

We now show that the reasonability property holds. Let $\pi : \m A \rightarrow \m B$ be an embedding between elements of $\Age(\m F)$, and $\m A^{*}$ an expansion of $\m A$ in $\mathcal K^{*}$. Because $\mathcal K^{*}=\Age(\m F, \vec P^{*}, <^{*})$, we can find a substructure of $(\m F, \vec P^{*}, <^{*})$ isomorphic to $\m A^{*}$; call it $\m A_{0}^{*}$, and let $\varphi : \m A^{*}\rightarrow \m A^{*}_{0}$ be an isomorphism. Then $\varphi$ induces an isomorphism from $\m A$ to $\restrict{\m A^{*}_{0}}{L}$, and by ultrahomogeneity of $\m F$, there is $g\in \Aut(\m F)$ such that $\restrict{g}{A}=\varphi$. Then $g(\m B)$ supports in $(\m F, \vec P^{*}, <^{*})$ an expansion $\m B^{*}$ of $\m B$. Pulling this back to $\m B$ via $g$, we see that $\m B^{*}$ is as required. \end{proof}

\subsection*{Acknowledgements}

I am grateful to Andr\'as Pongr\'acz, whose sharp eye caught a gap in the first proof of the main result of this paper, and to Claude Laflamme and Andy Zucker, for their comments on the first draft. I would also like to express my gratitude to Jan Hubi\v cka and Jaroslav Ne\v set\v ril for their invitation to lecture at the Ramsey DocCourse and publish this work. 

\bibliography{Bib17May}
\end{document}

%% file: Macros15Nov.tex
\newcounter{nthm}

\newcounter{nthmm}

\newtheorem{thm*}[nthm]{Theorem}
\newtheorem{thm**}[nthmm]{Th\'eor\`eme}
\newtheorem*{thm***}{Theorem}
\newtheorem{defn}{Definition}
\newtheorem*{defn*}{Definition}

\newtheorem{prop*}[nthm]{Proposition}

\newtheorem{lemma}{Lemma}

\newcommand{\N}{\mathbb{N}}

\newcommand{\restrict}[2]{#1 \upharpoonright #2}

\newcommand{\m}[1]{\textbf{#1}}

\newcommand{\mc}[1]{\widetilde{\textbf{#1}}}

\newcommand{\Aut}{\mathrm{Aut}}

\newcommand{\Age}{\mathrm{Age}}

%% file: RamseyDegrees2.bbl
\begin{bibdiv}
\begin{biblist}

\bib{BenYaacov16c}{unpublished}{
      author={{Ben Yaacov}, I.},
      author={Melleray, J.},
      author={Tsankov, T.},
       title={Metrizable universal minimal flows of polish groups have a
  comeagre orbit},
        date={2016},
        note={preprint},
}

\bib{Evans15}{unpublished}{
      author={Evans, D.~M.},
        date={2015},
        note={Personal communication},
}

\bib{Hubicka16}{unpublished}{
      author={Hubi{\v{c}}ka, J.},
      author={Ne{\v{s}}et{\v{r}}il, J.},
       title={{All those Ramsey classes}},
        date={2016},
        note={preprint},
}

\bib{Kriz91}{article}{
      author={K\v{r}\'\i\v{z}, I.},
       title={{Permutation groups in Euclidean Ramsey theory}},
        date={1991},
     journal={{Proc. Amer. Math. Soc.}},
      volume={112},
      number={3},
       pages={899\ndash 907},
}

\bib{Kechris05}{article}{
      author={Kechris, A.~S.},
      author={Pestov, V.~G.},
      author={Todorcevic, S.},
       title={Fra\"\i ss\'e limits, {R}amsey theory, and topological dynamics
  of automorphism groups},
        date={2005},
        ISSN={1016-443X},
     journal={Geom. Funct. Anal.},
      volume={15},
      number={1},
       pages={106\ndash 189},
}

\bib{Melleray16}{article}{
      author={Melleray, J.},
      author={{Nguyen Van Th{{\'e}}}, L.},
      author={Tsankov, T.},
       title={Polish groups with metrizable universal minimal flows},
        date={2016},
        ISSN={1073-7928},
     journal={Int. Math. Res. Not. IMRN},
      number={5},
       pages={1285\ndash 1307},
         url={http://dx.doi.org/10.1093/imrn/rnv171},
}

\bib{NVT13a}{article}{
      author={Nguyen Van~Th{\'e}, L.},
       title={More on the {K}echris-{P}estov-{T}odorcevic correspondence:
  precompact expansions},
        date={2013},
     journal={Fund. Math.},
      volume={222},
       pages={19\ndash 47},
}

\bib{Zucker16}{article}{
      author={Zucker, A.},
       title={Topological dynamics of automorphism groups, ultrafilter
  combinatorics, and the generic point problem},
        date={2016},
        ISSN={0002-9947},
     journal={Trans. Amer. Math. Soc.},
      volume={368},
      number={9},
       pages={6715\ndash 6740},
         url={http://dx.doi.org/10.1090/tran6685},
}

\end{biblist}
\end{bibdiv}
